\documentclass[12pt]{article}
\usepackage[letterpaper,lmargin=1.5in,height=8.3in,width=6in,top=1.5in]{geometry}
\usepackage{amsmath}
\usepackage{amsthm}
\usepackage{amssymb}
\usepackage{graphicx}
\usepackage{tkz-graph}
\usepackage{fancyhdr}
\newcommand{\NN}{\mathbb{N}}
\newcommand{\ZZ}{\mathbb{Z}}
\newcommand{\RR}{\mathbb{R}}
\newcommand{\QQ}{\mathbb{Q}}

\newcommand{\mb}[1]{\mathbf{#1}}
\newcommand{\mc}[1]{\mathcal{#1}}

\newcommand{\ol}[1]{\overline{#1}}
\newcommand{\wh}[1]{\widehat{#1}}
\newcommand{\wt}[1]{\widetilde{#1}}

\newcommand{\G}{\mathcal{G}}
\newcommand{\M}{\mathfrak{M}}
\newcommand{\Aut}{\mathrm{Aut}}
\newcommand{\U}{\mathcal{U}}

\newcommand{\e}{\varepsilon}

\newcommand{\ceil}[1]{\left\lceil #1 \right\rceil}
\newcommand{\floor}[1]{\left\lfloor #1 \right\rfloor}

\newtheorem{theo}{Theorem}[section]
\newtheorem{lem}[theo]{Lemma}
\newtheorem{prop}[theo]{Proposition}
\newtheorem{cor}[theo]{Corollary}
\theoremstyle{definition}
\newtheorem{defn}[theo]{Definition}

\begin{document}
\renewcommand{\thepage}{\roman{page}}

\thispagestyle{empty}

\begin{center}
\Large\upshape{Weak Convergence of Laws of Finite Graphs}
\end{center}

\vfill

\begin{center}
\large\rm Igor Artemenko
\end{center}

\vfill

\begin{center}
Project submitted in partial fulfilment of the requirements for the degree of B.Sc. Honours Specialisation in Mathematics
\end{center}

\vfill

\begin{center}
Department of Mathematics and Statistics\\
Faculty of Science\\
University of Ottawa
\end{center}

\vspace*{\fill}

\begin{center}
\includegraphics[width=1in]{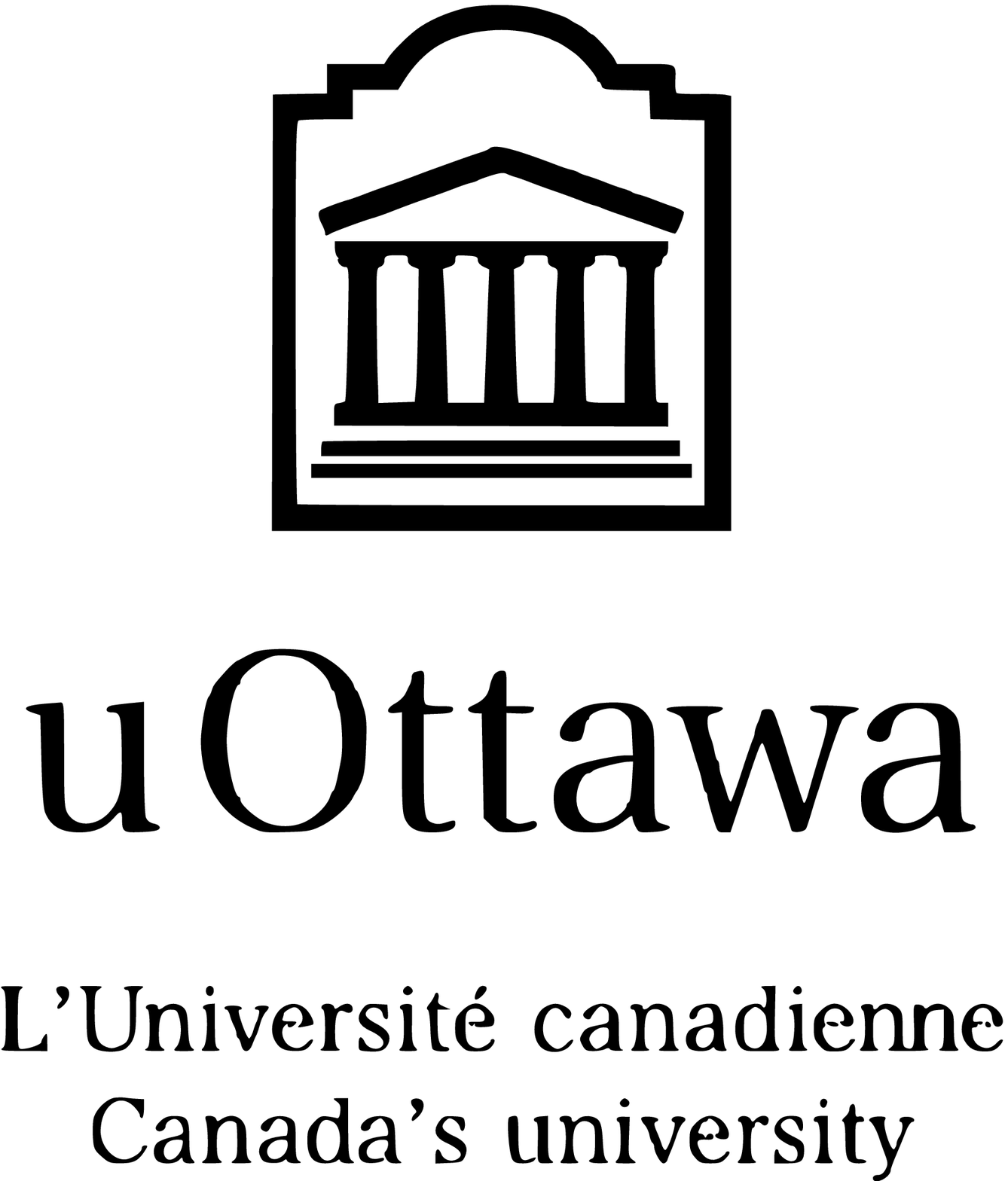}
\end{center}

\vfill

\begin{center}
\copyright\ Igor Artemenko, Ottawa, Canada, 2010
\end{center}
\clearpage

\thispagestyle{empty}
\tableofcontents

\thispagestyle{empty}
\addcontentsline{toc}{section}{List of Figures}
\listoffigures
\clearpage

\setcounter{page}{1}
\renewcommand{\thepage}{\arabic{page}}

\section{Introduction}
In recent years, several mathematicians have proceeded to study a concept known as unimodularity. This report is an introductory survey of this theory. The reader will encounter known results, new observations, and detailed examples. Our presentation is independent of any particular source, and we include rigorous proofs of facts whose demonstrations have been omitted elsewhere.

The results presented may very well lead to something nontrivial. In particular, with every proposition discovered and proved, we learn more about unimodularity, and so come closer to proving the current open problems and finding new ones.

Throughout this report, we will assume that the reader has a basic understanding of metric spaces and graph theory.

\section{Preliminaries}
Unless mentioned otherwise, the letters $k$, $m$, $n$, and their uppercase variants denote positive integers. Furthermore, $\ol{\NN} = \NN \cup \{\infty\}$ and $\NN^\ast = \NN \setminus \{0\}$.

All of the graphs in this paper are simple. That is, they have no loops, and there is at most one edge between any two vertices. A graph is \emph{finite} if its vertex set is finite. On the other hand, it is \emph{infinite} if it is not finite. If the vertex set of a graph is nonempty, the graph itself is said to be \emph{nonempty}. An edge in a graph is an unordered pair of vertices $\{u,v\}$, which we will denote by $uv$ instead.

In this report, a \emph{measure} on a space $X$ is a countably additive function from a $\sigma$-algebra of subsets of $X$ to the interval $[0,\infty]$. If $\mu$ is a measure on $X$ and $\mu(X) = 1$, then $\mu$ is a \emph{probability measure}. If a probability measure is defined on a subset of its domain, we extend it to a measure over the entire domain by setting it equal to zero elsewhere. The vertex set of every graph is equipped with the \emph{shortest path metric} $d$. That is, for all vertices $x$ and $y$ in a graph $G$, $d(x,y)$ is the length of the shortest path from $x$ to $y$; if no such path exists, $d(x,y) = \infty$.

\begin{defn}
Let $G$ be a graph; let $X$ be a subset of the vertices of $G$. A \emph{subgraph of $G$ induced by $X$} is a graph whose vertex set is $X$ and whose edge set is
\[
\{uv \in E(G) ~:~ u,v \in X\}.
\]
This subgraph is known as an \emph{induced subgraph} of $G$. It is denoted by $G[X]$.
\end{defn}

Note that the notion of an induced subgraph is stronger than that of a subgraph.

\begin{defn}
If $u$ and $v$ are vertices of a graph $G$, they are \emph{similar} if $\sigma(u) = v$ for some graph automorphism $\sigma$ on $G$. A graph is \emph{vertex-transitive} if any two of its vertices are similar.
\end{defn}

\begin{defn}
A $k$-\emph{rooted graph} is a $(k+1)$-tuple $(G,o_1,\ldots,o_k)$ where $G$ is a graph and $\{o_1,\ldots,o_k\} \subseteq V(G)$. Each $o_i$ is known as a \emph{root}; we do not assume that the roots are pairwise distinct. For convenience, we will refer to $1$-rooted graphs as rooted graphs, and $2$-rooted graphs as birooted graphs. Two $k$-rooted graphs $(G,o_1,\ldots,o_k)$ and $(G',o_1',\ldots,o_k')$ are \emph{isomorphic}, written
\[
(G,o_1,\ldots,o_k) \cong (G',o_1',\ldots,o_k'),
\]
if there exists a graph isomorphism $\varphi : G \to G'$ such that
\[
\varphi(o_i) = o_i'
\]
for each $i \in \{1,\ldots,k\}$.
\end{defn}

As the reader can verify, $\cong$ is an equivalence relation; its equivalence classes are called \emph{isomorphism classes}.

\begin{defn}
A graph is \emph{locally finite} if the degree of each of its vertices is finite.
\end{defn}

Let $\wh{\G}$ be the collection of all isomorphism classes of locally finite connected rooted graphs. The elements of $\wh{\G}$ are of the form $[G,o]$. The subcollection $\wh{\G}_M$ of $\wh{\G}$ consists of rooted graphs whose maximal degree is at most $M$. From this point on, we will simply write \emph{rooted graph} when referring to its isomorphism class.

Given a rooted graph $[G,o] \in \wh{\G}$ and a nonnegative real number $r$, let $B_G(o,r)$ be the subgraph of $G$ induced by the vertices at a distance of at most $r$ from $o$. That is,
\[
B_G(o,r) = G[B(o,r)]
\]
where $B(o,r) = \{v \in V(G) ~:~ d(o,v) \leq r\}$. The reader may think of the graphs $B_G(o,r)$ as closed balls.

If $G$ is a graph, then $G_x$ is the connected component of $G$ that contains $x$. The rooted graph $[G_x,x]$ is a \emph{rooted connected component} of $G$. Vertex-transitive graphs only have one rooted connected component, so it suffices to write $[G,\cdot]$ in this case. For clarity, we will write $(G_k,o_k)$ instead of $(G_{o_k},o_k)$. This notation will be especially useful when dealing with sequences of graphs.

When the reader encounters a figure of a rooted graph, the roots are the solid circles; all other vertices are empty circles.
\cleardoublepage

\section{The Metric $\rho$}
By itself, the set $\wh{\G}$ is not very interesting for our purposes. To benefit from it, we will equip it with a certain metric.

Given $[G,o],[G',o'] \in \wh{\G}$, we define the distance $\rho : \wh{\G} \times \wh{\G} \to \RR$ as follows:
\[
\rho([G,o],[G',o']) =
\begin{cases}
0               & \text{ if $[G,o] = [G',o']$,}\\
\frac{1}{1 + r} & \text{ otherwise.}
\end{cases}
\]
where
\[
r = \sup \{s \in \NN ~:~ [B_G(o,s),o] = [B_{G'}(o',s),o']\}.
\]

\begin{prop}
The function $\rho$ is a well-defined metric on $\wh{\G}$. Moreover, $\rho$ is an \emph{ultrametric}. That is, it satisfies a stronger version of the triangle inequality:
\[
\rho([G_1,o_1],[G_2,o_2]) \leq \max\{\rho([G_1,o_1],[G_3,o_3]),\rho([G_3,o_3],[G_2,o_2])\}
\]
for all $[G_1,o_1],[G_2,o_2],[G_3,o_3] \in \wh{\G}$.
\end{prop}

\begin{proof}
Let $[G_1,o_1],[G_1',o_1'],[G_2,o_2],[G_2',o_2'] \in \wh{\G}$ be arbitrary. Suppose that
\[
[G_1,o_1] = [G_2,o_2]
\]
and
\[
[G_1',o_1'] = [G_2',o_2'],
\]
and consider the following two cases.

\noindent\textbf{Case 1.} If $[G_1,o_1] = [G_1',o_1']$, then $[G_2,o_2] = [G_2',o_2']$, and so
\[
\rho([G_1,o_1],[G_1',o_1']) = 0 = \rho([G_2,o_2],[G_2',o_2']).
\]

\noindent\textbf{Case 2.} If $[G_1,o_1] \neq [G_1',o_1']$, then $[G_2,o_2] \neq [G_2',o_2']$. Let $r$ and $s$ be the largest integers such that
\[
[B_{G_1}(o_1,r),o_1] = [B_{G_1'}(o_1',r),o_1']
\]
and
\[
[B_{G_2}(o_2,s),o_2] = [B_{G_2'}(o_2',s),o_2'].
\]
Then
\[
[B_{G_1}(o_1,s),o_1] = [B_{G_2}(o_2,s),o_2] = [B_{G_2'}(o_2',s),o_2'] = [B_{G_1'}(o_1',s),o_1'],
\]
which means $s \leq r$. Similarly, $r \leq s$. Hence $r = s$, and so $\rho$ is well-defined.

To see that $\rho$ is a metric, it suffices to show that it satisfies the triangle inequality because the other conditions follow immediately from the definition. In fact, we will prove that
\[
\rho([G_1,o_1],[G_2,o_2]) \leq \max\{\rho([G_1,o_1],[G_3,o_3]),\rho([G_3,o_3],[G_2,o_2])\}.
\]
This inequality holds if any two of the three rooted graphs are equal. Suppose that
\[
\rho([G_1,o_1],[G_2,o_2]) = \frac{1}{1 + r},
\]
\[
\rho([G_1,o_1],[G_3,o_3]) = \frac{1}{1 + s},
\]
and
\[
\rho([G_3,o_3],[G_2,o_2]) = \frac{1}{1 + t}
\]
where $r$, $s$, and $t$ are the largest integers as defined previously. Since
\[
[B_{G_1}(o_1,s),o_1] = [B_{G_3}(o_3,s),o_3]
\]
and
\[
[B_{G_3}(o_3,t),o_3] = [B_{G_2}(o_2,t),o_2],
\]
it follows that
\[
[B_{G_1}(o_1,\min\{s,t\}),o_1] = [B_{G_2}(o_2,\min\{s,t\}),o_2].
\]
By definition, $\min\{s,t\} \leq r$, which means
\[
\frac{1}{1 + r} \leq \frac{1}{1 + \min\{s,t\}} = \max\left\{\frac{1}{1 + s},\frac{1}{1 + t}\right\},
\]
and the result follows.
\end{proof}

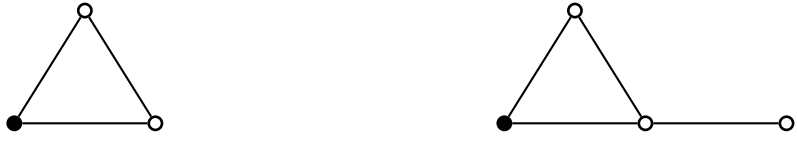
\begin{figure}
\begin{minipage}[b]{0.48\linewidth}
\begin{center}
\begin{tikzpicture}[scale=0.75]
\GraphInit[vstyle=Simple]
\Vertex[x=0,y=0,style={line width=1pt,fill=black,minimum size=5pt}]{A}
\Vertex[x=2.5,y=0,style={line width=1pt,fill=white,minimum size=5pt}]{B}
\Vertex[x=1.25,y=2,style={line width=1pt,fill=white,minimum size=5pt}]{C}
\Edges(A,B,C,A)
\end{tikzpicture}
\end{center}
\end{minipage}
\begin{minipage}[b]{0.48\linewidth}
\begin{center}
\begin{tikzpicture}[scale=0.75]
\GraphInit[vstyle=Simple]
\Vertex[x=0,y=0,style={line width=1pt,fill=black,minimum size=5pt}]{A}
\Vertex[x=2.5,y=0,style={line width=1pt,fill=white,minimum size=5pt}]{B}
\Vertex[x=1.25,y=2,style={line width=1pt,fill=white,minimum size=5pt}]{C}
\Vertex[x=5,y=0,style={line width=1pt,fill=white,minimum size=5pt}]{D}
\Edges(A,B,C,A)
\Edges(B,D)
\end{tikzpicture}
\end{center}
\end{minipage}
\caption[An example of distance]{An example of distance.}
\label{fig:example_of_distance}
\end{figure}

As with most concepts, examples are an excellent way to understand this metric space. Consider the graphs $[G,o]$ and $[H,p]$ shown in Figure \ref{fig:example_of_distance}. Then
\[
[B_G(o,1),o] = [B_H(p,1),p],
\]
and 1 is the largest integer for which this is true. Hence $\rho([G,o],[H,p]) = 1/2$. Note that, in either case, a ball of radius 2 is the entire graph.

When a metric space is defined, it is helpful to look at functions out of the space. An important example of such a function is given in the following proposition.

\begin{prop}\label{degree_function}
The \emph{degree function} $\deg$ defined by $[G,o] \mapsto \deg_G(o)$ for all $[G,o] \in \wh{\G}_M$ is $M$-Lipschitz.
\end{prop}

\begin{proof}
Let $[G,o],[G',o'] \in \wh{\G}_M$ be arbitrary. If $\deg_G(o) = \deg_{G'}(o')$, the result follows. Suppose that the two values are distinct. Then
\[
\rho([G,o],[G',o']) = 1,
\]
and so
\[
|\deg_G(o) - \deg_{G'}(o')| \leq M = M \rho([G,o],[G',o']). \qedhere
\]
\end{proof}

Note that the domain of the degree function is $\wh{\G}_M$, not all of $\wh{\G}$. Indeed this subspace is better-behaved than $\wh{\G}$. The theorem below reinforces this opinion.

\begin{theo}
The space $\wh{\G}_M$ is complete and totally bounded. That is, $\wh{\G}_M$ is compact.
\end{theo}

\begin{proof}[Proof of completeness.]
Let $([G_n,o_n])_{n=1}^\infty$ be a Cauchy sequence of graphs in $(\wh{\G}_M,\rho)$. For every positive integer $n$, define
\[
r_n = \sup\{r \in \NN ~:~ \forall k \in \NN ~~ [B_{G_n}(o_n,r),o_n] = [B_{G_{n+k}}(o_{n+k},r),o_{n+k}]\}.
\]
To simplify the notation, let $B_n = B_{G_n}(o_n,r_n)$. Furthermore, let
\[
(\mathcal{V},\mathcal{E}) = \left(\bigcup_{i=1}^\infty V(B_i) \times \{i\},\bigcup_{i=1}^\infty E(B_i) \times \{i\}\right).
\]

Note that there exists an isometric embedding $j_n : B_n \hookrightarrow B_{n+1}$, which we use to define an equivalence relation $\approx$ as follows:
\[
(u,n) \approx (v,m)
\]
if and only if
\[
j_{n+m} \circ j_{n+m-1} \circ \cdots \circ j_n(u) = j_{n+m} \circ j_{n+m-1} \circ \cdots \circ j_m(v)
\]
for all $(u,n),(v,m) \in \mc{V}$.

Following that, let $G$ be the graph with $V(G) = \mathcal{V}/\hspace{-0.5em}\approx$, and whose edge set is defined as follows:
\[
\{[u,n],[v,m]\} \in E(G)
\]
if and only if
\[
\{j_{n+m} \circ \cdots \circ j_n(u),j_{n+m} \circ \cdots \circ j_m(v)\} \in E(B_{n+m}).
\]

Having completed the construction of the rooted graph $[G,o]$ where $o = [o_1,1]$, we will now show that this is the limit of the Cauchy sequence defined at the beginning of the proof.

Throughout, we will assume that $\e \leq 1$ because the distance $\rho$ is bounded above by 1. Given $\e > 0$, there exists a nonnegative integer $N$ such that for all $k \in \NN$,
\[
\rho([G_N,o_N],[G_{N+k},o_{N+k}]) < \e
\]
because $([G_n,o_n])_{n=1}^\infty$ is Cauchy. If $n \geq N$, then
\begin{align*}
\rho([G_n,o_n],[G,o]) &\leq \max\{\rho([G_n,o_n],[G_N,o_N]),\rho([G_N,o_N],[G,o])\}\\
                      &< \max\left\{\e,\frac{1}{1 + r_N}\right\}.
\end{align*}
Let $R_k$ be the radius of the ball corresponding to $\rho([G_N,o_N],[G_{N+k},o_{N+k}])$. As stated above,
\[
\frac{1}{\e} - 1 < R_k
\]
for all $k \geq 0$. Consider the following cases. If
\[
\ceil{\frac{1}{\e} - 1} < R_k
\]
for all $k \geq 0$, then
\[
[B_{G_N}(o_N,\ceil{1/\e - 1} + 1),o_N] = [B_{G_{N+k}}(o_{N+k},\ceil{1/\e - 1} + 1),o_{N+k}]
\]
for all $k \geq 0$, and so
\[
\ceil{\frac{1}{\e} - 1} + 1 \leq r_N,
\]
which means $1/\e - 1 < r_N$. That is,
\[
\frac{1}{1 + r_N} < \e.
\]
On the other hand, suppose that
\[
\ceil{\frac{1}{\e} - 1} = R_{k_0}
\]
for some $k_0 \geq 0$. Since $\ceil{1/\e - 1} \leq R_k$,
\[
[B_{G_N}(o_N,R_{k_0}),o_N] = [B_{G_{N+k}}(o_{N+k},R_{k_0}),o_{N+k}]
\]
for all $k \geq 0$, and so
\[
R_{k_0} \leq r_N.
\]
By definition of $r_N$,
\[
[B_{G_N}(o_N,r_N),o_N] = [B_{G_{N+k_0}}(o_{N+k_0},r_N),o_{N+k_0}].
\]
Then
\[
\frac{1}{1 + R_{k_0}} \leq \frac{1}{1 + r_N},
\]
which means $r_N \leq R_{k_0}$. It follows that $R_{k_0} = r_N$. Then $1/\e < R_{k_0}$ implies that
\[
\frac{1}{1 + r_N} < \e.
\]
In either case,
\[
\rho([G_n,o_n],[G,o]) < \e.
\]
Thus $([G_n,o_n])_{n=1}^\infty$ converges to $[G,o]$.
\end{proof}

\begin{proof}[Proof of total boundedness.]
To show that $\wh{\G}_M$ is totally bounded, let $\e > 0$ and let
\[
r = \ceil{\frac{1}{\e} - 1}.
\]
Denote by $F$ the set of all rooted graphs in $\wh{\G}_M$ of radius at most $r$ where the radius is the supremum of the distances between each vertex and the root. Observe that $F$ is a finite set because any graph of radius $r$ has at most
\[
1 + M \sum_{i=1}^{r} (M - 1)^{i-1} \leq (M + 1)^r
\]
vertices, and so $|F|$ is at most the number of graphs on $(M + 1)^r$ or fewer vertices.

Now, given $[G,o] \in \wh{\G}_M$, the rooted graph $[B_G(o,r),o]$ is of radius $r$, and so it belongs to $F$. Furthermore,
\[
\rho([G,o],[B_G(o,r),o]) \leq \frac{1}{1 + r} \leq \e.
\]
Hence $\wh{\G}_M$ is totally bounded.
\end{proof}

The larger space $(\wh{\G},\rho)$ is \emph{not} compact. Consider the sequence of $(1,n)$-bipartite graphs $K_{1,n}$ each rooted at the vertex of degree $n$. Figure \ref{fig:star_graph_sequence} depicts a part of this sequence. This sequence has no convergent subsequence because the degree of the root is increasing. Note that such a sequence would not exist in the subspace $(\wh{\G}_M,\rho)$ where the degree of the root is at most $M$.

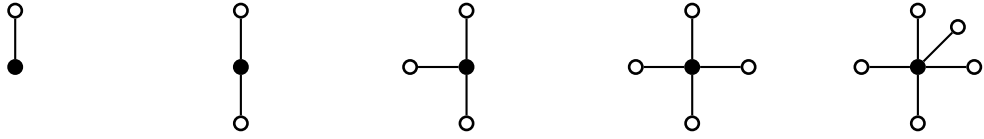
\begin{figure}
\begin{center}
\begin{tikzpicture}[scale=0.75]
\GraphInit[vstyle=Simple]
\Vertex[x=0,y=0,style={line width=1pt,fill=black,minimum size=5pt}]{A}
\Vertex[x=0,y=1,style={line width=1pt,fill=white,minimum size=5pt}]{A1}
\Vertex[x=4,y=0,style={line width=1pt,fill=black,minimum size=5pt}]{B}
\Vertex[x=4,y=1,style={line width=1pt,fill=white,minimum size=5pt}]{B1}
\Vertex[x=4,y=-1,style={line width=1pt,fill=white,minimum size=5pt}]{B2}
\Vertex[x=8,y=0,style={line width=1pt,fill=black,minimum size=5pt}]{C}
\Vertex[x=8,y=1,style={line width=1pt,fill=white,minimum size=5pt}]{C1}
\Vertex[x=8,y=-1,style={line width=1pt,fill=white,minimum size=5pt}]{C2}
\Vertex[x=7,y=0,style={line width=1pt,fill=white,minimum size=5pt}]{C3}
\Vertex[x=12,y=0,style={line width=1pt,fill=black,minimum size=5pt}]{D}
\Vertex[x=12,y=1,style={line width=1pt,fill=white,minimum size=5pt}]{D1}
\Vertex[x=12,y=-1,style={line width=1pt,fill=white,minimum size=5pt}]{D2}
\Vertex[x=11,y=0,style={line width=1pt,fill=white,minimum size=5pt}]{D3}
\Vertex[x=13,y=0,style={line width=1pt,fill=white,minimum size=5pt}]{D4}
\Vertex[x=16,y=0,style={line width=1pt,fill=black,minimum size=5pt}]{E}
\Vertex[x=16,y=1,style={line width=1pt,fill=white,minimum size=5pt}]{E1}
\Vertex[x=16,y=-1,style={line width=1pt,fill=white,minimum size=5pt}]{E2}
\Vertex[x=15,y=0,style={line width=1pt,fill=white,minimum size=5pt}]{E3}
\Vertex[x=17,y=0,style={line width=1pt,fill=white,minimum size=5pt}]{E4}
\Vertex[x=16.71,y=0.71,style={line width=1pt,fill=white,minimum size=5pt}]{E5}
\Edges(A,A1)
\Edges(B,B1)
\Edges(B,B2)
\Edges(C,C1)
\Edges(C,C2)
\Edges(C,C3)
\Edges(D,D1)
\Edges(D,D2)
\Edges(D,D3)
\Edges(D,D4)
\Edges(E,E1)
\Edges(E,E2)
\Edges(E,E3)
\Edges(E,E4)
\Edges(E,E5)
\end{tikzpicture}
\end{center}
\caption[The space $\wh{\G}$ is not compact]{The space $\wh{\G}$ is not compact.}
\label{fig:star_graph_sequence}
\end{figure}

Before moving to the next section, there are a few interesting results concerning continuous functions on $\wh{\G}_M$. Denote by $\wh{\G}_M^0$ the subspace of $\wh{\G}_M$ of finite rooted graphs.

\begin{prop}
The set $\wh{\G}_M^0$ is a dense subspace of $\wh{\G}_M$.
\end{prop}

\begin{proof}
Let $[G,o] \in \wh{\G}_M$ be arbitrary. Consider the sequence of finite rooted graphs defined by
\[
[G_n,o_n] = [B_G(o,n),o]
\]
for all positive integers $n$. This sequence converges to $[G,o]$, which means $[G,o]$ lies in the closure of $\wh{\G}_M^0$. The result follows.
\end{proof}

This fact is helpful when searching for continuous functions on the space of rooted graphs. By a result from analysis, a uniformly continuous function on $\wh{\G}_M^0$ induces a unique continuous extension on $\wh{\G}_M$. On the other hand, the next fact implies that all functions on $\wh{\G}_M^0$ are continuous--though not necessarily uniformly continuous.

\begin{lem}
The isolated points of $\wh{\G}_M$ are precisely the finite rooted graphs.
\end{lem}

\begin{proof}
If $[G,o] \in \wh{\G}_M$ is an isolated point, then any sequence that converges to it is eventually constant. In particular, this includes the sequence of balls rooted at $o \in V(G)$. Conversely, assume that $[G,o] \in \wh{\G}_M^0$, and let
\[
R = \inf\{r \in \NN ~:~ [G,o] = [B_G(o,r),o]\},
\]
which is the \emph{radius} of $[G,o]$. If $[G,o] \neq [H,p]$ for some $[H,p] \in \wh{\G}_M$, then
\[
\rho([G,o],[H,p]) = \frac{1}{1 + r}
\]
where $r$ is the largest nonnegative integer for which $[B_G(o,r),o] = [B_H(p,r),p]$. If $r > R$, the radius of $[H,p]$ must be at most $R$, but this means
\[
[G,o] = [B_G(o,R),o] = [B_H(p,R),p] = [H,p],
\]
which is a contradiction. Hence $r \leq R$, and so
\[
\rho([G,o],[H,p]) \geq \frac{1}{1 + R}.
\]
That is, $[G,o]$ is an isolated point.
\end{proof}

\begin{cor}
The space $\wh{\G}_M^0$ is topologically discrete.
\end{cor}

\section{The Law}
Let $\M_M$ be the space of all probability measures on $\wh{\G}_M$. Denote by $\mb{C}(\wh{\G}_M)$ the set of real-valued bounded continuous functions whose domain is $\wh{\G}_M$.

\begin{defn}
A sequence $(\mu_n)_{n=1}^\infty$ of probability measures on $\wh{\G}_M$ \emph{converges weakly} to some $\mu \in \M_M$ if
\[
\lim_{n \to \infty} \int f~d\mu_n = \int f~d\mu
\]
for all $f \in \mb{C}(\wh{\G}_M)$. The \emph{integral of $f$ with respect to the measure $\mu$} is the expression on the right-hand side, which may also be written as $\mu[f]$. Whenever we integrate a function, we will assume that it is measurable. The measure $\mu$ is known as the \emph{weak limit} of the given sequence. Occasionally, we will use $\mu_n \Rightarrow \mu$ to denote weak convergence of measures.
\end{defn}

Define $\G_M^0$ to be the collection of all isomorphism classes of nonempty finite graphs whose maximal degree is at most $M$. There is no difference between the notation of a finite graph and its isomorphism class, but this does not pose any confusion. For convenience, the elements of $\G_M^0$ will just be called finite graphs.

We equip the vertex set of every graph in $\G_M^0$ with the uniform probability measure. That is, if $G \in \G_M^0$, then the probability of choosing a vertex $v \in V(G)$ is $1/|V(G)|$. What follows is perhaps the most important definition in this paper.

\begin{defn}
The \emph{law} is a function $\Psi : \G_M^0 \to \M_M$ defined as follows: for every graph $G \in \G_M^0$,
\[
\Psi(G)[G_o,o] = \frac{|\Aut(G)o|}{|V(G)|}
\]
if $G_o$ is a connected component of $G$ for some $o \in V(G)$, and $\Psi(G) = 0$ elsewhere. Here $\Aut(G)$ is the group of automorphisms on $G$, and $\Aut(G)o$ is the \emph{orbit} of the vertex $o$ in $G$:
\[
\Aut(G)o = \{v \in V(G) ~:~ \exists \sigma \in \Aut(G) ~~ \sigma(v) = o\}.
\]
The image $\Psi(G)$ of a finite graph $G \in \G_M^0$ is a probability measure on $\wh{\G}_M$ called \emph{the law of $G$}. Usually, we will simply write \emph{the law} when no reference to a specific graph is necessary. The space of laws is the image $\Psi(\wh{\G}_M)$ equipped with the topology of weak convergence; its closure is denoted by $\M_M^0$.
\end{defn}

Occasionally, $|[G_o,o]|$ may be used to denote $|\Aut(G)o|$. This is justified by an equivalent definition of the law of a finite graph that arises from the following fact.

\begin{prop}
If $G \in \G_M^0$, then
\[
|\{(H,p) \in [G_o,o] ~:~ V(H) \subseteq V(G)\}| = |\Aut(G)o|
\]
whenever $[G_o,o]$ is a rooted connected component of $G$.
\end{prop}

\begin{proof}
Observe that $(H,p) \in [G_o,o]$ and $V(H) \subseteq V(G)$ imply that $H = G_p$. This means the function
\begin{align*}
f : \{(H,p) \in [G_o,o] ~:~ V(H) \subseteq V(G)\} &\to \Aut(G)o\\
                                            (H,p) &\mapsto p
\end{align*}
is injective. To see that $f$ is surjective, assume that $p \in \Aut(G)o$. Then $\sigma(p) = o$ for some automorphism $\sigma : G \to G$. Graph automorphisms are isometries because they preserve paths. If $r$ is a nonnegative integer, then
\begin{align*}
\sigma B(o,r) &= \{\sigma(x) \in V(G) ~:~ d(x,o) \leq r\}\\
              &= \{\sigma(x) \in V(G) ~:~ d(\sigma(x),p) \leq r\}\\
              &= \{y \in V(G) ~:~ d(y,p) \leq r\}\\
              &= B(p,r)
\end{align*}
because $\sigma$ is a surjective isometry. Hence $(B_G(o,r),o) \cong (B_G(p,r),p)$ for all nonnegative integers $r$. In particular, $(G_o,o) \cong (G_p,p)$. Since $V(G_p) \subseteq V(G)$, the rooted graph $(G_p,p)$ is an element of the domain of $f$. Furthermore, $f(G_p,p) = p$. Thus $f$ is a bijection, and so the domain and codomain have the same cardinality.
\end{proof}

Please note that $[G_o,o]$ is an infinite set, so $|[G_o,o]|$ is just a symbol that denotes the number
\[
|\{(H,p) \in [G_o,o] ~:~ V(H) \subseteq V(G)\}|.
\]
This notation is convenient though as demonstrated in the example below.

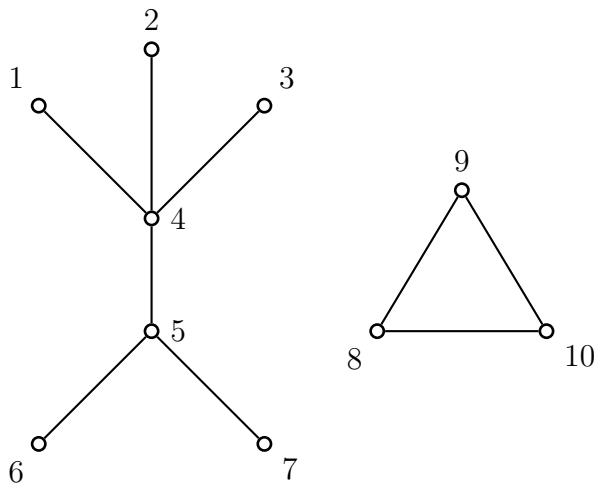
\begin{figure}
\begin{center}
\begin{tikzpicture}[scale=0.75]
\GraphInit[vstyle=Classic]
\Vertex[x=0,y=0,Lpos=240,style={line width=1pt,fill=white,minimum size=5pt}]{6}
\Vertex[x=4,y=0,Lpos=330,style={line width=1pt,fill=white,minimum size=5pt}]{7}
\Vertex[x=2,y=2,style={line width=1pt,fill=white,minimum size=5pt}]{5}
\Vertex[x=2,y=4,style={line width=1pt,fill=white,minimum size=5pt}]{4}
\Vertex[x=2,y=7,Lpos=90,style={line width=1pt,fill=white,minimum size=5pt}]{2}
\Vertex[x=0,y=6,Lpos=120,style={line width=1pt,fill=white,minimum size=5pt}]{1}
\Vertex[x=4,y=6,Lpos=60,style={line width=1pt,fill=white,minimum size=5pt}]{3}
\Edges(6,5)
\Edges(7,5)
\Edges(5,4)
\Edges(4,2)
\Edges(4,1)
\Edges(4,3)
\Vertex[x=6,y=2,Lpos=240,style={line width=1pt,fill=white,minimum size=5pt}]{8}
\Vertex[x=9,y=2,Lpos=330,style={line width=1pt,fill=white,minimum size=5pt}]{10}
\Vertex[x=7.5,y=4.5,Lpos=90,style={line width=1pt,fill=white,minimum size=5pt}]{9}
\Edges(8,10,9,8)
\end{tikzpicture}
\end{center}
\caption[The computation of a law]{A finite graph $G$ with 10 vertices.}
\label{fig:example_of_law}
\end{figure}

To visualize the definition of the law, consider the graph $G$ shown in Figure \ref{fig:example_of_law}. Note that $[G_6,6] = [G_7,7]$. In fact, $|[G_6,6]| = 2$. Furthermore, $[G_8,8] = [G_9,9] = [G_{10},10]$ and $|[G_8,8]| = 3$. Similarly, we compute that $\Psi(G)$ is defined by
\begin{align*}
\Psi(G)[G_1,1] &= 3/10\\
\Psi(G)[G_4,4] &= 1/10\\
\Psi(G)[G_5,5] &= 1/10\\
\Psi(G)[G_6,6] &= 2/10\\
\Psi(G)[G_8,8] &= 3/10.
\end{align*}

Before becoming too excited about $\Psi$, we must be certain that $\Psi(G)$ is actually a probability measure for every finite graph $G$. The following proposition establishes this certainty.

\begin{prop}
Let $G$ be a finite graph. The law $\Psi(G)$ is a well-defined probability measure.
\end{prop}

\begin{proof}
If $[G_o,o] = [G_{o'},o']$, then $\Aut(G)o = \Aut(G)o'$, and so $\Psi(G)$ is well-defined. Let $\{[G_1,o_1],\ldots,[G_k,o_k]\}$ be the set of rooted connected components of $G$. Recall that $\Aut(G)$ partitions $V(G)$, which means $V(G)$ is the disjoint union of the orbits of the roots of the rooted connected components. Then
\[
\Psi(G)(\wh{\G}_M) = \sum_{i=1}^k \Psi(G)[G_i,o_i] = \sum_{i=1}^k \frac{|\Aut(G)o_i|}{|V(G)|} = \frac{|V(G)|}{|V(G)|} = 1.
\]
Since $\Psi(G)$ has finite support, it is trivially countably additive. Hence $\Psi(G)$ is a probability measure.
\end{proof}

The remainder of this section will include various properties of $\Psi$. Most of the interesting results lie in the later sections, but some may use the propositions that follow.

\begin{prop}
The restriction of $\Psi$ to the set of graphs in $\G_M^0$ whose connected components are pairwise nonisomorphic is injective.
\end{prop}

\begin{proof}
Let $G$ and $H$ be graphs in $\G_M^0$ whose connected components are pairwise nonisomorphic. Suppose that $\Psi(G) = \Psi(H)$.

Given a root $o \in V(G)$, consider the rooted connected component $[G_o,o]$ of $G$. By definition of $\Psi(G)$, we know that $\Psi(G)[G_o,o] \neq 0$, and so $\Psi(H)[G_o,o] \neq 0$. Hence $[G_o,o]$ is a rooted connected component of $H$. That is, $[G_o,o] = [H_p,p]$ for some $p \in V(H)$.

Thus every rooted connected component of $G$ corresponds to a rooted connected component of $H$. Furthermore, $[H_p,p]$ is unique because the components in $H$ are pairwise nonisomorphic. It follows that $G$ is isomorphic to $H$.
\end{proof}

Note that $\Psi$ itself is not injective. For example, if $G$ consists of two disjoint copies of $K_3$ and $H$ is just one copy, then $\Psi(G) = \Psi(H)$, but the graphs $G$ and $H$ are not isomorphic.

The preceding proposition shows that to calculate the law of a certain graph, it suffices to consider only one copy of each component. A similar simplification occurs in the case of vertex-transitive graphs. A vertex-transitive graph $G \in \G_M^0$ has exactly one rooted connected component $[G_o,o]$. Its law is defined by $\Psi(G)[G_o,o] = 1$. That is, $\Psi(G)$ is the Dirac measure on the point $[G_o,o]$. Such measures will appear in a slightly more general context later on in the report.

We will use the next proposition several times to simplify the computation of integrals.

\begin{prop}
If $\{[G_1,o_1],\ldots,[G_k,o_k]\}$ is the set of the rooted connected components of a graph $G \in \G_M^0$, then
\[
\int f~d\Psi(G) = \frac{1}{|V(G)|} \sum_{x \in V(G)} f[G_x,x]
\]
for every real-valued function $f$ whose domain is $\wh{\G}_M$.
\end{prop}

\begin{proof}
Note that
\[
\sum_{x \in V(G)} f[G_x,x] = \sum_{i=1}^k f[G_i,o_i] \cdot |\Aut(G)o_i|.
\]
Hence
\begin{align*}
\int f~d\Psi(G) &= \sum_{i=1}^k f[G_i,o_i] \cdot \Psi(G)[G_i,o_i]\\
                &= \frac{1}{|V(G)|} \sum_{i=1}^k f[G_i,o_i] \cdot |\Aut(G)o_i|\\
                &= \frac{1}{|V(G)|} \sum_{x \in V(G)} f[G_x,x],
\end{align*}
as required.
\end{proof}

As there are now two types of convergence, the theorem below is a new observation that connects them when dealing with the special case of the Dirac measure.

\begin{theo}\label{weak_conv_implies_graph_conv}
If $(\Psi(G_n))_{n=1}^\infty$ converges weakly to the Dirac measure $\delta_{[G,o]}$ for some $[G,o] \in \wh{\G}_M$, then there exists a sequence of vertices $(o_n)_{n=1}^\infty$ such that $([G_n,o_n])_{n=1}^\infty$ converges\footnote{Here $G_n$ is the connected component of $o_n$. It is not the entire graph. That is, $[G_n,o_n]$ is actually $[(G_n)_{o_n},o_n]$. As the reader can see, this modification greatly simplifies the notation.} to $[G,o]$.
\end{theo}

\begin{proof}
For each positive integer $i$, choose $o_i \in V(G_i)$ such that
\[
\rho([G,o],[G_i,o_i])
\]
is minimized.

Recall that the function
\[
\rho_{[G,o]} : \wh{\G}_M \to \RR
\]
defined by
\[
\rho_{[G,o]}[G',o'] = \rho([G,o],[G',o'])
\]
for all $[G',o'] \in \wh{\G}_M$ is continuous. By assumption that $(\Psi(G_n))_{n=1}^\infty$ converges weakly to $\delta_{[G,o]}$,
\[
\lim_{n \to \infty} \int \rho_{[G,o]}~d\Psi(G_n) = \int \rho_{[G,o]}~d\delta_{[G,o]}.
\]
Note that
\[
\int \rho_{[G,o]}~d\delta_{[G,o]} = 0
\]
and
\[
\int \rho_{[G,o]}~d\Psi(G_n) = \sum_{x \in V(G_n)} \frac{\rho_{[G,o]}[(G_n)_x,x]}{|V(G_n)|}
\]
for each positive integer $n$. Hence the average value of $\rho_{[G,o]}$ converges to $0$, and so
\[
\rho([G,o],[G_n,o_n]) \to 0
\]
by definition of the sequence $(o_n)_{n=1}^\infty$.
\end{proof}

The reader may notice that the converse of this theorem is not true, but this will be demonstrated later on.

Something can be said about the convexity of the space of laws and its closure. To streamline the statements, a bit of notation is required. If $G$ and $H$ are graphs, then $G + H$ is their disjoint union. Similarly, $nG$ is a disjoint union of $n$ copies of $G$ for some nonnegative integer $n$. By convention, $0G$ is the graph with no vertices.

The next example and the proposition that follows put this new notation to good use. The measure defined by $\mu[K_1,\cdot] = 1/2$ and $\mu[K_2,\cdot] = 1/2$ yields a law. Indeed
\[
\mu = \Psi(2K_1 + K_2).
\]

\begin{prop}
If $m$ and $n$ are nonnegative integers with $(m,n) \neq (0,0)$, and $G$ and $H$ are finite graphs, then
\begin{align}\label{law_linearity}
\Psi(mG + nH) = \frac{m|V(G)|\Psi(G) + n|V(H)|\Psi(H)}{m|V(G)| + n|V(H)|}.
\end{align}
\end{prop}

\begin{proof}
Suppose that $[I_q,q]$ is a rooted connected component of $G$ and $H$. That is, $[I_q,q] = [G_o,o] = [H_p,p]$ for some $o \in V(G)$ and $p \in V(H)$. Then
\[
\Psi(mG + nH)[I_q,q] = \frac{|\Aut(mG + nH)q|}{|V(mG + nH)|} = \frac{|\Aut(mG + nH)q|}{m|V(G)| + n|V(H)|}.
\]
The cardinality of the orbit of $q$ in $mG + nH$ may be expressed as
\[
|\Aut(mG + nH)q| = m|\Aut(G)o| + n|\Aut(H)p|.
\]
The result is
\[
\Psi(mG + nH)[I_q,q] = \frac{m|\Aut(G)o| + n|\Aut(H)p|}{m|V(G)| + n|V(H)|}.
\]
On the other hand,
\[
\frac{m|V(G)|\Psi(G)[G_o,o] + n|V(H)|\Psi(H)[H_p,p]}{m|V(G)| + n|V(H)|} = \frac{m|\Aut(G)o| + n|\Aut(H)p|}{m|V(G)| + n|V(H)|}.
\]
Hence Equation \ref{law_linearity} holds. Similar reasoning applies to the three remaining cases where $[I_q,q]$ is a rooted connected component of $G$ or $H$, but not both.
\end{proof}

As the following lemma and theorem demonstrate, the line segment between a pair of weak limits of laws lies in the closure of the space of laws.

\begin{lem}
The set of laws is rationally convex. That is, for all $G,H \in \G_M^0$ and $t \in [0,1] \cap \QQ$, the measure $t\Psi(G) + (1 - t)\Psi(H)$ is a law.
\end{lem}

\begin{proof}
Let $t = p/q$ for some integers $p$ and $q$ with $q \neq 0$. By Equation \ref{law_linearity},
\begin{align*}
\Psi(p|V(H)|G &+ (q - p)|V(G)|H)\\
              &= \frac{p|V(H)||V(G)|\Psi(G) + (q - p)|V(G)||V(H)|\Psi(H)}{p|V(H)||V(G)| + (q - p)|V(G)||V(H)|}\\
              &= \left(\frac{p}{q}\right)\Psi(G) + \left(\frac{q - p}{q}\right)\Psi(H)\\
              &= t\Psi(G) + (1 - t)\Psi(H).
\end{align*}
Hence the right-hand side is the law of the graph $p|V(H)|G + (q - p)|V(G)|H$.
\end{proof}

%
%

\begin{theo}
The closure of the set of laws is convex.
\end{theo}

\begin{proof}
Let $\mu_G$ and $\mu_H$ be limits of laws: $\Psi(G_n) \Rightarrow \mu_G$ and $\Psi(H_n) \Rightarrow \mu_H$. Given a real number $a$, there exists a sequence $a_n$ of rational numbers such that $a_n \to a$.

Suppose that $f$ is a continuous function on $\wh{\G}_M$. Note that
\[
\int f~d(a_n\Psi(G_n) + (1 - a_n)\Psi(H_n)) = a_n\int f~d\Psi(G_n) + (1 - a_n)\int f~d\Psi(H_n)
\]
for all positive integers $n$, and
\[
\int f~d(a\mu_G + (1 - a)\mu_H) = a\int f~d\mu_G + (1 - a)\int f~d\mu_H.
\]
Then
\begin{align*}
\Bigg|\bigg(a_n\int f ~d\Psi(G_n) &+ (1 - a_n)\int f ~d\Psi(H_n)\bigg) - \bigg(a\int f ~d\mu_G + (1 - a)\int f ~d\mu_H\bigg)\Bigg|\\
                                  &\leq |a_n\Psi(G_n)[f] - a\mu_G[f]| + |(1 - a_n)\Psi(H_n)[f] - (1 - a)\mu_H[f]|.
\end{align*}
The greater term is at most
\begin{align*}
|\Psi(G_n)[f]||a_n - a| &+ |a||\Psi(G_n)[f] - \mu_G[f]|\\
                        &+ |a_n - a||\Psi(H_n)[f]| + |1 - a||\Psi(H_n)[f] - \mu_H[f]|,
\end{align*}
which tends to zero. Hence
\[
\lim_{n \to \infty} \int f ~d(a_n\Psi(G_n) + (1 - a_n)\Psi(H_n)) = \int f ~d(a\mu_G + (1 - a)\mu_H),
\]
and so $a\mu_G + (1 - a)\mu_H$ is the weak limit of a sequence of laws.
\end{proof}

\section{Unimodularity}
Having defined the law of a finite graph, the reader may be curious to see why these maps are useful.

\begin{defn}
A measure $\mu$ on $\wh{\G}_M$ is \emph{unimodular} if
\[
\int \sum_{x \in V(G)} f[G,x,o]~d\mu[G,o] = \int \sum_{x \in V(G)} f[G,o,x]~d\mu[G,o]
\]
for all nonnegative functions $f$ whose domain is the set of birooted connected graphs. Such a measure is also said to satisfy the \emph{intrinsic Mass Transport Principle}, the iMTP. The set of these measures is denoted by $\U$.
\end{defn}

The following fact appears in a paper by Schramm \cite{hgl}, but its proof is omitted.

\begin{prop}\label{laws_are_unimodular}
Every law is unimodular.
\end{prop}

\begin{proof}
Suppose that $G$ is a finite graph. Denote by $\omega$ the number of connected components of $G$. If $\omega = 1$, then
\begin{align*}
\int \sum_{x \in V(G)} f[G,x,o]~d\Psi(G)[G,o] &= \frac{1}{|V(G)|} \sum_{y \in V(G)} \sum_{x \in V(G)} f[G,x,y]\\
                                              &= \frac{1}{|V(G)|} \sum_{x \in V(G)} \sum_{y \in V(G)} f[G,x,y]\\
                                              &= \int \sum_{y \in V(G)} f[G,o,y]~d\Psi(G)[G,o]
\end{align*}
for all rooted connected components $[G,o]$ of $G$. Hence the result holds for connected graphs. Suppose that $\omega = 2$. Write $G = H + I$ where $H$ and $I$ are its connected components. By Equation \ref{law_linearity},
\[
\Psi(G) = \frac{|V(H)|\Psi(H) + |V(I)|\Psi(I)}{|V(G)|}.
\]
Since $\Psi(H)$ and $\Psi(I)$ are unimodular and the integral is linear, it follows that $\Psi(G)$ is unimodular. The general statement holds by induction on $\omega$.
\end{proof}

Aldous and Lyons \cite{pourn}, and Schramm \cite{hgl} conjectured that every unimodular measure is the weak limit of a sequence of laws, but this remains open.

\section{Sustained Probability Measures}
Although laws are interesting objects, their domain consists only of finite graphs. This section attempts to view laws in a slightly more general setting.

\begin{defn}
A probability measure $\mu$ on $\wh{\G}_M$ is \emph{sustained} by a graph $G$ if the support of $\mu$ is a subset of rooted connected components of $G$. It is \emph{strictly} sustained by $G$ if every rooted connected component has a positive measure.
\end{defn}

\begin{lem}\label{imtp_sustained_implies_strictly_sustained}
If a unimodular measure $\mu \in \U$ is sustained by a finite connected graph $G$, then $\mu$ is strictly sustained by $G$.
\end{lem}

\begin{proof}
Let $\{[G,i] ~:~ 1 \leq i \leq k\}$ be the set of rooted connected components of $G$, and let $\mu[G,i] = p_i$ for all $i \in \{1,\ldots,k\}$. Suppose that $p_j = 0$ for some $j \in \{1,\ldots,k\}$. Consider the function
\[
f[H,a,b] =
\begin{cases}
1 & \text{ if $[H,a] = [G,j]$,}\\
0 & \text{ otherwise,}
\end{cases}
\]
which is well-defined. Observe that
\[
\sum_{x \in V(G)} f[G,x,i] = |\Aut(G)j|
\]
for all $i \in \{1,\ldots,k\}$. Since $\mu$ is unimodular,
\begin{align*}
|\Aut(G)j| &= \sum_{i=1}^k p_i |\Aut(G)j|\\
           &= \sum_{i=1}^k \sum_{x \in V(G)} f[G,x,i] p_i\\
           &= \int \sum_{x \in V(G)} f[G,x,o]~d\mu[G,o]\\
           &= \int \sum_{x \in V(G)} f[G,o,x]~d\mu[G,o]\\
           &= \sum_{i=1}^k \sum_{x \in V(G)} f[G,i,x] p_i\\
           &= 0,
\end{align*}
which is a contradiction.
\end{proof}

\begin{lem}\label{arithmetic_harmonic_means}
Let $\{a_1,\ldots,a_n\}$ be a set of positive real numbers. If $\sum_{i=1}^n a_i = 1$ and $\sum_{i=1}^n a_i^{-1} = n^2$, then $a_i = n^{-1}$ for all $i \in \{1,\ldots,n\}$.
\end{lem}

\begin{proof}
Recall that
\[
\frac{\sum_{i=1}^n a_i}{n}
\]
is the arithmetic mean of the given set, and
\[
\frac{n}{\sum_{i=1}^n a_i^{-1}}
\]
is its harmonic mean. In this case, both means are equal to $n^{-1}$. It is known that the arithmetic and harmonic means coincide if and only if
\[
a_1 = \cdots = a_n,
\]
and so $a_i = n^{-1}$ for all $i \in \{1,\ldots,n\}$.
\end{proof}

\begin{theo}\label{law_characterization}
If $\mu \in \M_M$ is sustained by a finite connected graph $G$ and satisfies the iMTP, then $\mu$ is the law of $G$.
\end{theo}

\begin{proof}
Let $\{[G,i] ~:~ 1 \leq i \leq k\}$ be the set of rooted connected components of $G$, and let $\mu[G,i] = p_i$ for all $i \in \{1,\ldots,k\}$. The functions
\[
f : [G,x,y] \mapsto \frac{p_x}{|[G,x]|}
\]
and
\[
h : [G,x,y] \mapsto \frac{p_x|[G,y]|}{p_y|[G,x]|}
\]
are nonnegative and measurable. Since $\mu$ satisfies the iMTP,
\begin{align*}
1 &= \left(\sum_{i=1}^k p_i\right)\left(\sum_{j=1}^k p_j\right)\\
  &= \sum_{i=1}^k \sum_{x \in V(G)} \left(\frac{p_x}{|[G,x]|}\right) p_i\\
  &= \int \sum_{x \in V(G)} f[G,x,o]~d\mu\\
  &= \int \sum_{x \in V(G)} f[G,o,x]~d\mu\\
  &= \sum_{i=1}^k \sum_{x \in V(G)} \frac{p_i^2}{|[G,i]|}\\
  &= |V(G)|\sum_{i=1}^k \frac{p_i^2}{|[G,i]|},
\end{align*}
and so
\[
\sum_{i=1}^k \frac{p_i^2}{|[G,i]|} = \frac{1}{|V(G)|}.
\]
Using this equation and similar reasoning,
\begin{align*}
               &\int \sum_{x \in V(G)} h[G,x,o]~d\mu = \int \sum_{x \in V(G)} h[G,o,x]~d\mu\\
\Rightarrow ~~ &\sum_{i=1}^k \sum_{x \in V(G)} h[G,x,i]p_i = \sum_{i=1}^k \sum_{x \in V(G)} h[G,i,x]p_i\\
\Rightarrow ~~ &\sum_{i=1}^k \sum_{x \in V(G)} \frac{p_x|[G,i]|}{|[G,x]|} = \sum_{i=1}^k \sum_{x \in V(G)} \frac{p_i^2|[G,x]|}{p_x|[G,i]|}\\
\Rightarrow ~~ &\left(\sum_{i=1}^k |[G,i]|\right)\left(\sum_{x \in V(G)} \frac{p_x}{|[G,x]|}\right) = \left(\sum_{i=1}^k \frac{p_i^2}{|[G,i]|}\right)\left(\sum_{x \in V(G)} \frac{|[G,x]|}{p_x}\right)\\
\Rightarrow ~~ &|V(G)| = \frac{1}{|V(G)|} \sum_{x \in V(G)} \frac{|[G,x]|}{p_x},
\end{align*}
which means
\[
\sum_{x \in V(G)} \frac{|[G,x]|}{p_x} = |V(G)|^2.
\]
By applying Lemma \ref{arithmetic_harmonic_means} on the set
\[
\left\{\frac{p_x}{|[G,x]|} ~:~ x \in V(G)\right\},
\]
it follows that
\[
p_x = \frac{|[G,x]|}{|V(G)|}
\]
for all $x \in V(G)$, as required.
\end{proof}

As stated in Proposition \ref{laws_are_unimodular}, laws are unimodular. Since laws themselves are sustained probability measures, the converse of Theorem \ref{law_characterization} is also true. Hence this theorem is a new characterization of laws of finite connected graphs. Aldous and Lyons \cite{pourn} briefly mention this characterization without proof.

A silly but instructive application of Lemma \ref{imtp_sustained_implies_strictly_sustained} is the following proposition.

\begin{prop}
If the Dirac measure $\delta_{[G,o]}$ is unimodular for some $[G,o] \in \wh{\G}_M$, then $G$ is vertex-transitive.
\end{prop}

\begin{proof}
Suppose that $p \notin \Aut(G)o$ for some $p \in V(G)$. Then $[G,o] \neq [G,p]$, and so $\delta_{[G,o]}[G,p] = 0$. By the assumption that $\delta_{[G,o]}$ is unimodular and Lemma \ref{imtp_sustained_implies_strictly_sustained}, this Dirac measure is strictly sustained by $G$; a contradiction.
\end{proof}

\section{The Space of Paths}
This section will cover the case of $M = 2$. Specifically, we will concentrate on a special subspace of $\wh{\G}_2$. Even though this is only a small portion of the general theory, the insight we gain will be helpful in understanding $\wh{\G}_M$.

\begin{defn}
A \emph{finite rooted path} is an element of the set
\[
\left\{[P_k,o] \in \wh{\G}_2 ~:~ k \in \NN^\ast\right\}
\]
where $P_k$ is a path of length $k - 1$ with $k$ vertices. These may also be denoted by $[P(u,v),o]$ where $u$ and $v$ are the ends of the path $P$.

Consider the graph whose vertex set is $\ZZ$, and in which consecutive integers are adjacent. Such a graph is known as the \emph{bi-infinite} path $P_\infty$.

An initial segment $[N,\infty) \cap \ZZ$ where $N$ is an integer induces the rooted graph $[P(N,\infty),0]$, which is called a \emph{semi-infinite} rooted path.

Note that $P_\infty$ is vertex-transitive, meaning it has exactly one rooted connected component, which we will denote by $[P_\infty,\cdot]$. An instance of each of these rooted graphs is shown in Figure \ref{fig:types_of_paths}.
\end{defn}

\begin{figure}
\begin{center}
\begin{tikzpicture}[scale=0.75]
\GraphInit[vstyle=Classic]
\Vertex[NoLabel,x=0,y=4,style={line width=1pt,fill=white,minimum size=5pt}]{A}
\Vertex[Lpos=90,x=2,y=4,style={line width=1pt,fill=black,minimum size=5pt}]{$u_2$}
\Vertex[NoLabel,x=4,y=4,style={line width=1pt,fill=white,minimum size=5pt}]{C}
\Vertex[NoLabel,x=6,y=4,style={line width=1pt,fill=white,minimum size=5pt}]{D}
\Vertex[NoLabel,x=8,y=4,style={line width=1pt,fill=white,minimum size=5pt}]{E}
\Edges(A,$u_2$,C,D,E)
\Vertex[Lpos=90,x=0,y=2,style={line width=1pt,fill=white,minimum size=5pt}]{$-2$}
\Vertex[Lpos=90,x=2,y=2,style={line width=1pt,fill=white,minimum size=5pt}]{$-1$}
\Vertex[Lpos=90,x=4,y=2,style={line width=1pt,fill=black,minimum size=5pt}]{$0$}
\Vertex[Lpos=90,x=6,y=2,style={line width=1pt,fill=white,minimum size=5pt}]{$1$}
\Vertex[Lpos=90,x=8,y=2,style={line width=1pt,fill=white,minimum size=5pt}]{$2$}
\Vertex[NoLabel,x=10,y=2,style={line width=1pt,fill=white,minimum size=0pt}]{K}
\Edges($-2$,$-1$,$0$,$1$,$2$)
\Edges[style={dotted}]($2$,K)
\Vertex[NoLabel,x=-2,y=0,style={line width=1pt,fill=white,minimum size=0pt}]{L}
\Vertex[NoLabel,x=0,y=0,style={line width=1pt,fill=white,minimum size=5pt}]{M}
\Vertex[NoLabel,x=2,y=0,style={line width=1pt,fill=white,minimum size=5pt}]{N}
\Vertex[NoLabel,x=4,y=0,style={line width=1pt,fill=black,minimum size=5pt}]{O}
\Vertex[NoLabel,x=6,y=0,style={line width=1pt,fill=white,minimum size=5pt}]{P}
\Vertex[NoLabel,x=8,y=0,style={line width=1pt,fill=white,minimum size=5pt}]{Q}
\Vertex[NoLabel,x=10,y=0,style={line width=1pt,fill=white,minimum size=0pt}]{R}
\Edges(M,N,O,P,Q)
\Edges[style={dotted}](L,M)
\Edges[style={dotted}](Q,R)
\end{tikzpicture}
\end{center}
\caption[Three types of rooted paths]{The rooted paths $[P_5,u_2]$, $[P(-2,\infty),0]$, and $[P_\infty,\cdot]$.}
\label{fig:types_of_paths}
\end{figure}
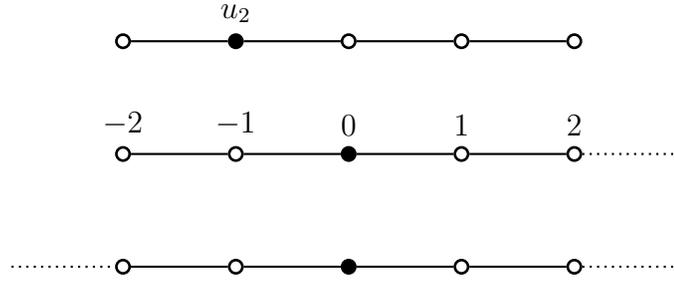

\begin{prop}
The semi-infinite and bi-infinite rooted paths are the limit points of the set of finite rooted paths.
\end{prop}

\begin{proof}
As the reader can verify, each of the mentioned rooted graphs is indeed a limit point. On the other hand, $[G,o]$ is a limit point only if there is a sequence of finite rooted paths that converges to $[G,o]$. Using the same construction as in the proof of completeness of the space $\wh{\G}_M$, the limit of the finite rooted paths will be a rooted path, bi-infinite or semi-infinite.
\end{proof}

With the proposition above, we can define the special subspace mentioned at the beginning of this section. The \emph{space of paths} $\mc{P}$ is the closure of the set of finite rooted paths.

\subsection{Basics and Properties}
One of the main attractions of the finite rooted paths is the ability to calculate the law explicitly.

\begin{prop}
Let $k$ be a nonnegative integer. Write $P_{2k} = u_1 \cdots u_{2k}$ and $P_{2k+1} = u_1 \cdots u_{2k+1}$. Then
\[
\Psi(P_{2k})[P_{2k},u_i] = \frac{1}{k}
\]
for all $i \in \{1,\ldots,k\}$ and
\[
\Psi(P_{2k+1})[P_{2k+1},u_i] =
\begin{cases}
\frac{2}{2k + 1} & \text{ if $i \in \{1,\ldots,k\}$},\\
\frac{1}{2k + 1} & \text{ if $i = k + 1$}.
\end{cases}
\]
\end{prop}

\begin{proof}
Given a nonnegative integer $n$, let $P_n = u_1 \cdots u_n$. Then
\[
[P_n,u_i] = [P_n,u_{n-i+1}]
\]
for all $i \in \{1,\ldots,\floor{n/2}\}$, and the result follows.
\end{proof}

To illustrate the usefulness of these explicit formulae, consider the following example. Let $\mu_n = \Psi(P_n)$ and let $\mu$ be the Dirac measure on the bi-infinite path rooted at some vertex. Recall that the degree function $\deg$ defined in Proposition \ref{degree_function} is continuous. Then
\begin{align*}
\int \deg~d\mu_{2k} &= \frac{\sum_{i=1}^k \deg[P_{2k},u_i]}{k}\\
                    &= \frac{\sum_{i=1}^k \deg_{P_{2k}}(u_i)}{k}\\
                    &= \frac{2k - 1}{k}\\
                    &= 2 - \frac{1}{k}
\end{align*}
and
\[
\int \deg~d\mu = \deg[P_\infty,\cdot] = 2.
\]
Hence
\[
\lim_{k \to \infty} \int \deg~d\mu_{2k} = \int \deg~d\mu.
\]
In fact this is not only true for the function $\deg$, but for all continuous functions whose domain is $\wh{\G}_M$. Indeed we will soon demonstrate this result.

A slightly stronger variation of something that has already been mentioned is summarized in the following lemma. Using this lemma, we will show that the sequence of laws of finite paths converges to the Dirac measure on the bi-infinite path.

\begin{lem}\label{convergence_of_paths}
The sequence $\left([P_{2k},u_k]\right)_{k=1}^\infty$ converges to $[P_\infty,\cdot]$. Furthermore,
\[
\rho\left([P_{2(k+l)},u_k],[P_\infty,\cdot]\right) = \rho\left([P_{2k},u_k],[P_\infty,\cdot]\right)
\]
for every nonnegative integer $l$.
\end{lem}

\begin{prop}
The sequence $(\mu_{2k})_{k=1}^\infty$ converges weakly to $\mu$.
\end{prop}

\begin{proof}
Let $g : \wh{\G}_M \to \RR$ be an arbitrary continuous function with $|g| \leq B$ for some real number $B$. Given $\e > 0$, there exists $\delta > 0$ such that for all $[G,o] \in \wh{\G}_M$,
\[
\rho\left([G,o],[P_\infty,\cdot]\right) < \delta ~~ \Rightarrow ~~ |g[G,o] - g[P_\infty,\cdot]| < \frac{\e}{2}.
\]
By Lemma \ref{convergence_of_paths}, there exists $N_1$ such that
\[
\forall k \geq N_1 ~~ \forall l \geq 0 ~~ \rho\left([P_{2(k+l)},u_k],[P_\infty,\cdot]\right) < \delta.
\]
Furthermore, there exists $N_2$ such that
\[
\forall k \geq N_2 ~~ \frac{1}{k} < \frac{\e}{4(M + 1)N_1}.
\]
Choose $N = \max\{N_1,N_2\}$. Then
\begin{align*}
\left|\int g~d\mu_k - \int g~d\mu\right| &=    \frac{1}{k}\left|\sum_{i=1}^k \left(g[P_{2k},u_i] - g[P_\infty,\cdot]\right)\right|\\
                                         &\leq \frac{1}{k}\sum_{i=1}^k \left|g[P_{2k},u_i] - g[P_\infty,\cdot]\right|\\
                                         &\leq \frac{1}{k}\sum_{i=1}^{N_1-1} \left|g[P_{2k},u_i] - g[P_\infty,\cdot]\right| + \frac{1}{k}\sum_{i=N_1}^k \left|g[P_{2k},u_i] - g[P_\infty,\cdot]\right|\\
                                         &<    \frac{2M(N_1 - 1)}{k} + \frac{k - N_1}{k} \cdot \frac{\e}{2}\\
                                         &<    \frac{\e}{2} + \frac{\e}{2}\\
                                         &=    \e
\end{align*}
for all $k \geq N$. Hence $\mu_{2k} \Rightarrow \mu$.
\end{proof}

\begin{cor}
The same is true for the sequence $(\mu_{2k+1})_{k=1}^\infty$.
\end{cor}

The two facts above may be combined using the following lemma.

\begin{lem}
Let $(a_n)_{n=1}^\infty$ be a sequence of real numbers. If $a_{2k} \to a$ and $a_{2k+1} \to a$, then $a_n \to a$.
\end{lem}

\begin{proof}
Given $\e > 0$, there exists $N_1$ such that for all $k \geq N_1$, $|a_{2k} - a| < \e$, and $N_2$ such that for all $k \geq N_2$, $|a_{2k+1} - a| < \e$. Choose $N = \max\{2N_1,2N_2 + 1\}$. Then $|a_n - a| < \e$ for all $n \geq N$.
\end{proof}

\begin{theo}\label{weak_limit_of_paths}
The Dirac measure on $[P_\infty,\cdot]$ is the weak limit of $(\Psi(P_n))_{n=1}^\infty$.
\end{theo}

Something that is much more surprising is that this is the only possible limit in the case of rooted paths.

\begin{prop}
Suppose that $\mu$ is a probability measure in $\M_M$ whose support is a subset of $\mc{P}$. If $(\mu_n)_{n=1}^\infty$ is a sequence of laws of finite paths that weakly converges to $\mu$, then either this sequence is eventually constant or it contains a subsequence of $(\Psi(P_n))_{n=1}^\infty$.
\end{prop}

\begin{proof}
Suppose that $(\mu_n)_{n=1}^\infty$ is not eventually constant. Define a function $f$ as follows: for every $n \geq 1$, there exists an $m \geq 1$ such that $\mu_n = \Psi(P_m)$; let $f(n) = m$.

We will show that
\[
\forall N ~~ \exists n > N ~~ f(N) < f(n).
\]
To derive a contradiction, assume that there exists $N$ such that $f(N) \geq f(n)$ for all $n > N$. It follows that
\[
(\mu_k)_{k=N+1}^\infty = (\Psi(P_{f(k)}))_{k=N+1}^\infty
\]
has only a finite number of distinct terms. Since this sequence also converges to $\mu$, we know it is eventually constant, and so $(\mu_n)_{n=1}^\infty$ is too; a contradiction.

Using this fact, if $k = 1$, choose $n_1 > 1$ such that $f(1) < f(n_1)$; then choose $n_2 > n_1$ such that $f(n_1) < f(n_2)$; and so forth. With this construction,
\[
n_1 < n_2 < \cdots
\]
and
\[
f(n_1) < f(n_2) < \cdots
\]
which means $(\mu_{n_k})_{k=1}^\infty$ is a subsequence of $(\mu_n)_{n=1}^\infty$, and $(\Psi(P_{f(n_k)}))_{k=1}^\infty$ is a subsequence of $(\Psi(P_m))$.
\end{proof}

\begin{cor}
If $\M_M^0(\mc{P})$ is the set of measures in $\M_M^0$ whose support is a subset of $\mc{P}$, then
\[
\M_M^0(\mc{P}) = \Psi(\{P_n ~:~ n \in \NN^\ast\}) \cup \{\delta_{[P_\infty,\cdot]}\}.
\]
\end{cor}

Essentially, this means the space of rooted paths is not interesting. Indeed to better understand the concept of the iMTP, a space that offers more variety with regard to measures is necessary. Before we consider such a space, we will attempt to demonstrate why paths themselves are still intriguing.

\subsection{A Natural Representation of $\mc{P}$}
Although the title of this subsection includes the word ``natural,'' nothing explicitly categorical is happening. Its use is justified because the space of rooted paths may be viewed as a subspace of the plane $\ol{\NN}^2$. What we construct is an explicit metric space model of $\mc{P}$. To simplify our arguments, we will adopt the following convention:
\[
\frac{1}{1 + \infty} = 0.
\]

Let $\wt{\mc{P}} = \{(x,y) \in \ol{\NN}^2 ~:~ x \leq y\}$. Define the distance $\tilde{\rho} : \wt{\mc{P}} \times \wt{\mc{P}} \to \RR$ as follows:
\[
\tilde{\rho}\left((x,y),(x',y')\right) = \frac{1}{1 + r}
\]
where
\[
r =
\begin{cases}
\infty            & \text{ if $x =    x'$ and $y =    y'$},\\
\min\{x,x'\}      & \text{ if $x \neq x'$ and $y =    y'$},\\
\min\{y,y'\}      & \text{ if $x =    x'$ and $y \neq y'$},\\
\min\{x,x',y,y'\} & \text{ if $x \neq x'$ and $y \neq y'$}
\end{cases}
\]
for all $(x,y),(x',y') \in \wt{\mc{P}}$.

\begin{theo}
The metric space $(\mc{P},\rho)$ is isometric to $(\wt{\mc{P}},\tilde{\rho})$.
\end{theo}

\begin{proof}
Define a function $f : \mc{P} \to \wt{\mc{P}}$ by
\[
f[P(u,v),o] = \left(\min\{d(u,o),d(o,v)\},\max\{d(u,o),d(o,v)\}\right)
\]
for all $[P(u,v),o] \in \mc{P}$. It suffices to show that $f$ is an isometry, a well-defined surjective function that preserves distances. It is straightforward to verify that $f$ is well-defined and preserves distances. To see that $f$ is surjective, let $(x,y) \in \wt{\mc{P}}$ be arbitrary. Consider the path defined by the sequence of integers
\[
(-x,-x+1,\ldots,-1,0,1,\ldots,y-1,y).
\]
The isomorphism class of this path when rooted at zero is mapped to $(x,y)$, as required.
\end{proof}

By keeping this result in mind, we may use the notation of natural numbers instead of paths and vertices. For convenience, we will not distinguish between $\mc{P}$ and $\wt{\mc{P}}$ for the remainder of this section. The following example uses this simpler notation when deriving a formula for the measure of a vertical strip in the space $\mc{P}$.

A \emph{vertical strip} is a set of the form
\[
A_m = \{(x,y) \in \mc{P} ~:~ m \leq x\}
\]
where $m$ is a positive integer with $m \leq k$. Observe that
\begin{align*}
\Psi(P_{2k+1})(A_m) &= \sum_{i=m}^{k-1} \Psi(P_{2k+1})(i,2k - i) + \Psi(P_{2k+1})(k,k)\\
                    &= \sum_{i=m}^{k-1} \frac{2}{2k + 1} + \frac{1}{2k + 1}\\
                    &= \frac{2(k - m)}{2k + 1} + \frac{1}{2k + 1}\\
                    &= \frac{2k + 1 - 2m}{2k + 1}\\
                    &= 1 - \frac{2m}{2k + 1}
\end{align*}
and
\begin{align*}
\Psi(P_{2k})(A_m) &= \sum_{i=m}^{k-1} \Psi(P_{2k})(i,2k - 1 - i)\\
                  &= \sum_{i=m}^{k-1} \frac{1}{k}\\
                  &= \frac{k - m}{k}\\
                  &= 1 - \frac{m}{k}.
\end{align*}

In general,
\[
\Psi(P_n)(A_m) =
\begin{cases}
1 - \frac{2n}{m} & \text{ if $m \leq \floor{\frac{n}{2}}$,}\\
0                & \text{ otherwise}
\end{cases}
\]
for all positive integers $m$ and $n$.

\begin{lem}\label{finite_paths_are_discrete}
The subspace $(\NN^2,\rho)$ is topologically discrete. That is, every point is isolated.
\end{lem}

\begin{proof}
If $(x,y) \in \NN^2$, then
\[
\{(x,y)\} = B_\rho\left((x,y),\min\left\{\frac{1}{x + 1},\frac{1}{y + 1}\right\}\right).
\]
To see this, suppose that
\[
\rho((x,y),(x',y')) < \min\left\{\frac{1}{x + 1},\frac{1}{y + 1}\right\}.
\]
If $x = x'$ and $y \neq y'$, then
\[
\rho((x,y),(x',y')) = \frac{1}{\min\{y,y'\} + 1} \geq \frac{1}{y + 1},
\]
which is a contradiction. A similar argument establishes the other cases.
\end{proof}

\begin{theo}
As a topological space, $(\mc{P},\rho)$ is homeomorphic to $(\alpha\NN^2,d_\infty)$ where $d_\infty$ is the $\ell_\infty$ distance and
\[
\alpha\NN = \left\{\frac{1}{n} ~:~ n \in \NN^\ast\right\} \cup \{0\}.
\]
\end{theo}

\begin{proof}
Define the function $f : \ol{\NN}^2 \to \alpha\NN^2$ by
\[
f(x,y) = \left(\frac{1}{x + 1},\frac{1}{y + 1}\right)
\]
for all $(x,y) \in \ol{\NN}^2$. We will show that $f$ is continuous. It is known that every function is continuous at the isolated points in its domain. By Lemma \ref{finite_paths_are_discrete}, $f$ is continuous on $\NN^2$.

Let $x \in \NN$ and let $\e > 0$. Choose
\[
\delta = \min\left\{\frac{1}{x + 1},\e\right\}.
\]
Suppose that $(x',y') \in \ol{\NN}^2 \setminus \{(x,\infty)\}$. The only case in which $\rho((x,\infty),(x',y')) < \delta$ is when $x = x'$ and $y' \in \NN$. Then
\begin{align*}
d_\infty(f(x,\infty),f(x',y')) &= \max\left\{\left|\frac{1}{x + 1} - \frac{1}{x' + 1}\right|,\left|0 - \frac{1}{y' + 1}\right|\right\}\\
                               &= \frac{1}{y' + 1}\\
                               &= \rho((x,\infty),(x',y'))\\
                               &< \delta\\
                               &\leq \e.
\end{align*}
Similarly, the same is true for $(\infty,y)$ for all $y \in \NN$. To see that $f$ is continuous at $(\infty,\infty)$, note that
\begin{align*}
d_\infty(f(\infty,\infty),f(x',y')) &= \max\left\{\frac{1}{x' + 1},\frac{1}{y' + 1}\right\}\\
                                    &= \frac{1}{\min\{x',y'\} + 1}\\
                                    &= \rho((\infty,\infty),(x',y'))
\end{align*}
if $(x',y') \neq (\infty,\infty)$. In the case of equality, the result follows as well.

Hence $f$ is continuous on $\ol{\NN}^2$. Furthermore, it is a bijection. Upon applying a result from topology, which states that a continuous bijection from a compact space to a Hausdorff space is a homeomorphism, the proof is complete.
\end{proof}

Before leaving this section, the reader may be interested to know what else is in the space $\wh{\G}_2$. Although it is somewhat surprising, it is not very difficult to prove that
\[
\wh{\G}_2 = \mc{P} \cup \{[C_n,\cdot] ~:~ n \geq 3\}.
\]
In fact, $([C_n,\cdot])_{n=3}^\infty$ converges to $[P_\infty,\cdot]$.

\section{The Trees Inside $\wh{\G}_3$}
As we have seen in the previous section, the laws of paths are not interesting enough. To remedy this situation, we will move to--in some sense--a higher dimension.

Throughout this section, the object of study will be the 3-regular infinite tree $T_\infty$. Since $T_\infty$ is vertex-transitive, we will fix a root $t$ and use it for the remainder of this section. For each $n$, let
\[
T_n = B_{T_\infty}(t,n).
\]
These trees will mimic the finite paths seen in the previous section.

Recall Theorem \ref{weak_conv_implies_graph_conv}, which related the weak convergence of laws to the convergence of rooted graphs. As mentioned then, the converse of this theorem is not true. Indeed this section yields a simple counterexample.

The sequence $([T_n,t])_{n=0}^\infty$ converges to $[T_\infty,t]$, but $\Psi(T_n) \not \Rightarrow \delta_{[T_\infty,t]}$. To see this, let $\deg$ be the degree function defined earlier. Suppose that $\Psi(T_n) \Rightarrow \delta_{[T_\infty,t]}$. In particular,
\[
\lim_{n \to \infty} \int \deg~d\Psi(T_n) = \int \deg~d\delta_{[T_\infty,t]}.
\]
Note that
\[
\int \deg~d\Psi(T_n) = \frac{1}{|V(T_n)|} \sum_{x \in V(T_n)} \deg_{T_n}(x) = \frac{2|E(T_n)|}{|V(T_n)|} = 2 - \frac{2}{|V(T_n)|}
\]
for all nonnegative integers $n$, and so
\[
\lim_{n \to \infty} \int \deg~d\Psi(T_n) = 2.
\]
On the other hand,
\[
\int \deg~d\delta_{[T_\infty,t]} = \deg[T_\infty,t] = 3;
\]
a contradiction.

Recall that there is a unique path between two vertices in a tree. With this in mind, let $u_1 \cdots u_k t$ be such a path in $T_k$ from a leaf $u_1$ to the vertex $t$.

\begin{theo}\label{trees_converge_to_s}
The sequence $(\Psi(T_k))_{k=0}^\infty$ converges weakly to a measure $\mu$ on $\wh{\G}_M$ defined by
\[
\mu[S,u_i] = \frac{1}{2^i}
\]
where $[S,u_i]$ is the limit of the sequence $([T_k,u_i])_{k=i}^\infty$ for all positive integers $i$.
\end{theo}

\begin{proof}
Let $f \in \mb{C}(\wh{\G}_M)$ be arbitrary. Since $f$ is bounded, there exists a real number $L$ such that
\[
\sup\left\{|f[G,o]| \in \RR ~:~ [G,o] \in \wh{\G}_M\right\} \leq L.
\]
To simplify the proof, observe that
\begin{align*}
\lim_{k \to \infty} \int f ~d\Psi(T_k) &= \lim_{k \to \infty} \left(\sum_{i=1}^k \frac{f[T_k,u_i] \cdot |[T_k,u_i]|}{|V(T_k)|} + \frac{f[T_k,t]}{|V(T_k)|}\right)\\
                                       &= \lim_{k \to \infty} \left(\sum_{i=1}^k \frac{f[T_k,u_i] \cdot 3 \cdot 2^{k-i}}{3 \cdot 2^k - 2} + \frac{f[T_k,t]}{3 \cdot 2^k - 2}\right)\\
                                       &= \lim_{k \to \infty} \left(\frac{3 \cdot 2^k}{3 \cdot 2^k - 2}\right) \lim_{k \to \infty} \left(\sum_{i=1}^k \frac{f[T_k,u_i]}{2^i}\right) + \lim_{k \to \infty} \left(\frac{f[T_k,t]}{3 \cdot 2^k - 2}\right)\\
                                       &= \lim_{k \to \infty} \left(\sum_{i=1}^k \frac{f[T_k,u_i]}{2^i}\right).
\end{align*}

Now it suffices to prove that the expression in the previous line is the integral of $f$ with respect to $\mu$. To demonstrate this, let $\e > 0$. We know that
\[
\frac{L}{2^N} < \frac{\e}{4}
\]
for some positive integer $N$. By definition of $[S,u_i]$, the sequence $(f[T_k,u_i])_{k=i}^\infty$ converges to $f[S,u_i]$ for each $i \in \{1,\ldots,N\}$ because $f$ is continuous. That is, there exists a positive integer $n_i$ such that
\[
|f[T_k,u_i] - f[S,u_i]| < \frac{\e}{2}
\]
for all $k \geq \max\{n_i,N + 1\}$.

Choose $n = \max\{n_1,n_2,\ldots,n_N,N + 1\}$ and let $k$ be an integer with $k \geq n$. Then
\begin{align*}
\Bigg|\sum_{i=1}^k \frac{f[T_k,u_i]}{2^i} &- \sum_{i=1}^\infty \frac{f[S,u_i]}{2^i}\Bigg|\\
                                          &\leq \sum_{i=1}^N \frac{|f[T_k,u_i] - f[S,u_i]|}{2^i} + \sum_{i=N+1}^k \frac{|f[T_k,u_i]|}{2^i} + \sum_{i=N+1}^\infty \frac{|f[S,u_i]|}{2^i}\\
                                          &< \sum_{i=1}^N \frac{\e}{2^{i+1}} + \sum_{i=N+1}^k \frac{L}{2^i} + \sum_{i=N+1}^\infty \frac{L}{2^i}\\
                                          &\leq \left(\frac{\e}{2}\right) \sum_{i=1}^\infty \frac{1}{2^i} + \frac{2L}{2^N}\\
                                          &< \frac{\e}{2} + \frac{\e}{2}\\
                                          &= \e,
\end{align*}
and so
\[
\lim_{k \to \infty} \int f ~d\Psi(T_k) = \lim_{k \to \infty} \left(\sum_{i=1}^k \frac{f[T_k,u_i]}{2^i}\right) = \sum_{i=1}^\infty \frac{f[S,u_i]}{2^i} = \int f ~d\mu.
\]

Hence $(\Psi(T_k))_{k=0}^\infty$ converges weakly to $\mu$.
\end{proof}

\begin{figure}
\begin{center}
\begin{tikzpicture}[scale=0.75]
\GraphInit[vstyle=Classic]
\Vertex[Lpos=-90,x=0,y=0,style={line width=1pt,fill=white,minimum size=5pt}]{$u_1$}
\Vertex[Lpos=-90,x=4,y=0,style={line width=1pt,fill=white,minimum size=5pt}]{$u_2$}
\Vertex[Lpos=-90,x=8,y=0,style={line width=1pt,fill=white,minimum size=5pt}]{$u_3$}
\Vertex[Lpos=-90,x=12,y=0,style={line width=1pt,fill=white,minimum size=5pt}]{$u_4$}
\Vertex[NoLabel,x=16,y=0,style={line width=1pt,fill=white,minimum size=0pt}]{A}
\Edges($u_1$,$u_2$,$u_3$,$u_4$)
\Edges[style={dotted}]($u_4$,A)
\Vertex[NoLabel,x=4,y=2.5,style={line width=1pt,fill=white,minimum size=5pt}]{B}
\Vertex[NoLabel,x=8,y=2.5,style={line width=1pt,fill=white,minimum size=5pt}]{C}
\Vertex[NoLabel,x=6.9,y=4,style={line width=1pt,fill=white,minimum size=5pt}]{C1}
\Vertex[NoLabel,x=9.1,y=4,style={line width=1pt,fill=white,minimum size=5pt}]{C2}
\Vertex[NoLabel,x=12,y=2.5,style={line width=1pt,fill=white,minimum size=5pt}]{D}
\Vertex[NoLabel,x=10.9,y=4,style={line width=1pt,fill=white,minimum size=5pt}]{D1}
\Vertex[NoLabel,x=13.1,y=4,style={line width=1pt,fill=white,minimum size=5pt}]{D2}
\Vertex[NoLabel,x=10.2,y=5,style={line width=1pt,fill=white,minimum size=5pt}]{D11}
\Vertex[NoLabel,x=11.6,y=5,style={line width=1pt,fill=white,minimum size=5pt}]{D12}
\Vertex[NoLabel,x=12.4,y=5,style={line width=1pt,fill=white,minimum size=5pt}]{D21}
\Vertex[NoLabel,x=13.8,y=5,style={line width=1pt,fill=white,minimum size=5pt}]{D22}
\Edges($u_2$,B)
\Edges($u_3$,C)
\Edges($u_4$,D)
\Edges(C,C1)
\Edges(C,C2)
\Edges(D,D1)
\Edges(D,D2)
\Edges(D1,D11)
\Edges(D1,D12)
\Edges(D2,D21)
\Edges(D2,D22)
\end{tikzpicture}
\end{center}
\caption[The graph $S$]{The graph $S$.}
\label{fig:the_graph_S}
\end{figure}
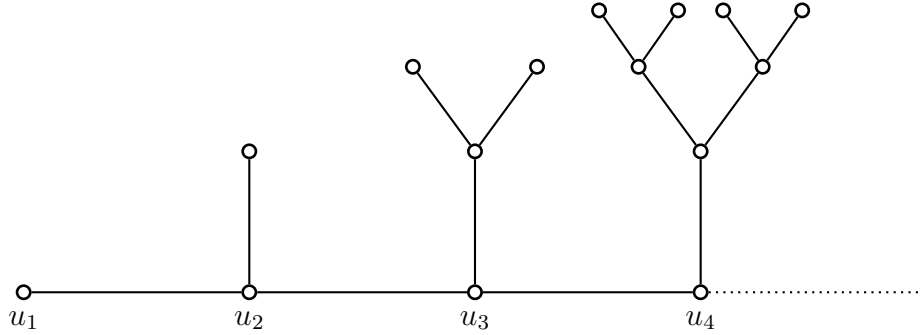

The graph $S$ defined above is illustrated in Figure \ref{fig:the_graph_S}.
\cleardoublepage

\section{Hyperfinite Collections}
Although we will not go into detail, Schramm \cite{hgl} discusses the concept of a hyperfinite collection whose definition is presented below. In this section, we merely provide an example of such a collection.

\begin{defn}
A collection of finite graphs $\G$ is $(k,\e)$-\emph{hyperfinite} for some positive integer $k$ and $\e > 0$ if for every $G \in \G$, there exists a set $S \subseteq E(G)$ such that $|S| \leq \e|V(G)|$, and each connected component of $G \setminus S$ has at most $k$ vertices.

The collection $\G$ is \emph{hyperfinite} if for every $\e > 0$, there is a positive integer $k$ such that $\G$ is $(k,\e)$-hyperfinite.
\end{defn}

The collection $\mc{S}$ of finite paths is hyperfinite. Given $\e > 0$, choose
\[
k = \ceil{1/\e}.
\]
Let $P_{n+1} = e_1 \cdots e_n$ be a finite path for some positive integer $n$ written as a sequence of its edges. Consider the set $S = \{e_i \in E(P_{n+1}) ~:~ k | i\}$. Then
\[
|S| = \floor{\frac{n}{k}},
\]
and every connected component of $P_{n+1} \setminus S$ has at most $k$ vertices.

\begin{prop}
The union of a finite number of hyperfinite collections is hyperfinite.
\end{prop}

\section{Possible Directions of Research}
There is much left to uncover in the realm of unimodular measures.

Although it was not shown in this report, it is true that the Dirac measure $\delta_{[P_\infty,\cdot]}$ satisfies the iMTP. Indeed the infinite path may be ``rotated'' by 180 degrees, which means $[P_\infty,x,y]$ is symmetric in $x$ and $y$. A more general question that we wish to answer is whether a Dirac measure on \emph{any} vertex-transitive graph is unimodular. Note that Dirac measures of finite vertex-transitive graphs are laws, which we know now are unimodular, so the question is directed at infinite graphs.

We may also consider the open problem mentioned earlier that claims every unimodular measure is a weak limit of a sequence of laws. Elek \cite{otlolggs} has provided a partial solution to this problem, and we wish to understand his reasoning. In fact we have yet to prove the converse of this conjecture, which, according to Aldous and Lyons \cite{pourn}, and Schramm \cite{hgl}, is trivial.

Perhaps the convexity of the closure of laws offers interesting results as well. Hence its further study is also warranted.
\cleardoublepage

\addcontentsline{toc}{section}{References}

\cleardoublepage
\end{document}